\documentclass[a4paper,times]{amsart}
\usepackage[colorlinks]{hyperref}
\usepackage{caption}
\usepackage{tikz}
\usepackage{pgfplots}
\usepackage{amsmath,amssymb,amsthm,eucal,upref}
\usepackage{changes}

\newtheorem{theorem}{Theorem}[section]
\newtheorem{proposition}[theorem]{Proposition}
\newtheorem{lemma}[theorem]{Lemma}
\newtheorem{corollary}[theorem]{Corollary}

\newtheorem*{theorem*}{Theorem}

\theoremstyle{definition}

\theoremstyle{remark}
\newtheorem{remark}[theorem]{\rm\bf Remark}
\usepackage{soul}

\begin{document}

\title{Bifurcations and the exchange of stability with coinfection}

\author{J. Andersson and V. Kozlov and V.G.~Tkachev and U. Wennergren}
	\begin{abstract}
	We perform a bifurcation analysis on an SIR model involving two pathogens that influences each other. Partial cross-immunity is assumed and coinfection is thought to be less transmittable then each of the diseases alone. The susceptible class has density dependent growth with carrying capacity $K$. Our model generalises the model developed in \cite{part1} and \cite{part2} by introducing the possibility for coinfected individuals to spread only one of the diseases when in contact with a susceptible. We perform a bifurcation analysis and prove the existence of a branch of stable equilibrium points parmeterized by $K$. The branch bifurcates for some $K$ resulting in changes in which compartments are present as well as the overall dynamics of the system. Depending on the parameters different transition scenarios occur.
	\end{abstract}


\maketitle

\section{Introduction}
Mathematical models for infectious diseases are valuable as a tools for predicting the future dynamics of diseases as well as finding the best preventative strategies. Usually the mathematical model involves only one pathogen, though in many cases two or more pathogens of different variants or strain can influence the spread of each other. Also the health effect of being infected by several diseases at once can be much more detrimental than being infected by one disease at a time.

In \cite{SKTW18a} an SIR model involving two pathogens that can co-infect is developed. The model takes into consideration density dependent growth of the susceptible population but ignores it for the other compartments. The density dependence is modelled with a logistic growth term which can be modelled with a carrying capacity parameter $K$. For simplicity infected individuals are assumed to not give birth. This assumption is motivated, even if infected individuals can give birth, when the disease compartments are small. The model can even work if the disease compartments are relatively large. This is because the logistic growth term makes it so that any flow out of the susceptible compartment, including individuals being infected, increases the generation of new susceptibles i.e. births. This process compensate for the ignored birth in the disease compartments. In this simple model the only way for an individual to become coinfected is by being infected by a coinfected individual.

The papers \cite{part1} and \cite{part2}, which are a partition of a larger paper, improves on the model developed in \cite{SKTW18a} by introducing the possibility for two individuals infected by different pathogens or strains to infect each other and thus become coinfected. Finally in \cite{GKTW} the model is developed further by introducing the possibility for healthy individuals to only be infected by one of the diseases when getting in contact with a coinfected individual. In all of the papers \cite{GKTW}, \cite{part1}, \cite{part2} the introduced transmission rates, which all distinct themselves from the rest of the transmission rates by describing interactions between two compartments that in turn increases a third compartment, are assumed to be small. This constraint is due to our bifurcation approach, but can also have epidemiological motivations such as strong cross immunity.

In \cite{SKTW18a} it was discovered that there exists a branch of uniquely stable equilibrium points parametrized by $K$. The way the equilibrium point of this branch transitions between different types is determined by the parameters other than $K$ and it turns out that there are four different scenarios (By the type of the equilibrium point we are referring to which of the compartments are non-zero). For this model an exact description could be found for the branch.

In \cite{part1} and \cite{part2} we get a similar branch to the one found in \cite{SKTW18a} that only differs by a term that is the order of magnitude of the newly introduces parameters which are considered to be small. The transitions between different types are basically the same as in \cite{SKTW18a}. However we get a deviation for the scenario where the equilibrium point of the branch is an inner equilibrium for large $K$ in this case the branch can lose its stability for large $K$, this results in a Hopf bifurcation.

In \cite{GKTW} we get a different scenario for the branch. In this case, the existence of some types of equilibrium points that in previous papers where possible are now impossible since contrary to \cite{SKTW18a}, \cite{part1}, \cite{part2} single infections now necessarily must exist if coinfection is present. From \cite{GTKW} we know that there is a branch of uniquely stable equilibrium points starting of as disease free point that transitions into a point where the more transmissible disease is present. There exist a point $K=K_0$ where this branch loses its stability and it is believed that, similar to previous models, the branch bifurcated into a stable coexistence equilibrium point.

Their exist many relevant mathematical studies on the interaction between multiple strains or diseases, see for instance \cite{ci1},\cite{ci2},\cite{ci3},\cite{ci4},\cite{ci5},\cite{ci6}.

In this paper we proceed with the work in \cite{GKTW} and prove that the bifurcation into a uniquely stable coexistence point indeed occurs. We also derive conditions for stability of the inner equilibrium point.  In contrast to the other equilibrium points the expression of the inner equilibrium is very implicit, and it is impossible to draw any conclusion by direct methods. Instead, we develop a certain bifurcation technique to obtain our main result.

%

\section{The model and the main result}


Let us consider the system appearing in a study of SIR model with coinfections, see \cite{GKTW}. The model describes the progression of two transmittable diseases, denoted disease 1 and disease 2, infecting a population of susceptibles with size $S(t)$, more precisely
\begin{eqnarray}
&&\Big(r(1-\frac{S}{K})-\alpha_1I_1 -\alpha_2I_2 -(\beta_1 +\beta_2 + \alpha_3)I_{12}\Big)S=S',\label{SIRmodelrow1}\\
&& (\alpha_1S-\eta_1I_{12} -\gamma_1I_2 -\mu_1)I_1 +\beta_1SI_{12}=I'_1\label{SIRmodelrow2}\\
&&(\alpha_2S-\eta_2I_{12}-\gamma_2I_1 -\mu_2)I_2 +\beta_2SI_{12}=I'_2,\label{SIRmodelrow3}\\
&&( \alpha_3S+ \eta_1I_1 +\eta_2I_2 -\mu_3)I_{12} +( \gamma_1 +\gamma_2)I_1I_2=I'_{12}\label{SIRmodelrow4}\\
&&\rho_1 I_1+\rho_2I_2+\rho_3 I_{12}-d_4 R=R',\label{submodel2}
\end{eqnarray}
where we use the following notation:
\begin{itemize}
\item $I_1$ and $I_2$ are infected classes from strain 1 and strain 2 respectively;
\item $I_{12}$ is the compartment consisting of individuals infected by both diseases;
\item $R$ represents the recovered class.
\end{itemize}
Also, \textit{complete cross immunity is assumed}, meaning if an individual becomes immune to one of the diseases it will necessarily be immune to both diseases.

Following \cite{Allen,Bremermann,Zhou}, we shall assume a limited population growth by making the per capita growth rate depend on the density of the population. We also consider different recovery rates for each infected class (see \eqref{submodel2}). The fundamental  parameters of the system are:
\begin{itemize}
\item $b$ is the birthrate,
\item
$K$ is the carrying capacity,
\item
$\rho_i$ is the recovery rate for each infected class ($i=1,2,3$),
\item
$\gamma_i$ is the rate at which individuals infected with one strain get infected with  the other strain from single infected individuals and move to the coinfected class ($i=1,2$).
\item $\mu_i' $ is the death rate of each class,  $( i=1,2,3,4)$,
\item
$\alpha_1$, $\alpha_2$, $\alpha_3$  are the rates of transmission of strain 1, strain 2 and both strains (in the case of coinfection),
\item $\beta_i$ is the rate at which susceptiles contracts disease $i$ but not the other disease from a coinfected individuals  $i=1,2$

\item and
$\eta_i$ is the rate at which individuals infected by one strain gets coinfected by the coinfected class $( i=1,2)$ and $\mu_i=\rho_i+\mu_i', i=1,2,3.$
\end{itemize}

\tikzset{
  big arrow/.style={
    decoration={markings,mark=at position 1 with {\arrow[scale=2,#1]{>}}},
    postaction={decorate},
    shorten >=0.4pt}}
\def\FIG{
\begin{figure}[h]
	\begin{center}
		\begin{tikzpicture}[every node/.style={ minimum height={1cm},minimum width={2cm},thick,align=center}]

		\node[draw] (I1) {$I_1$};
		\node[draw, above right=of I1] (S) {$S$};
		\node[draw, below=of S] (I12) {$I_{12}$};
		\node[draw, below=of I12] (R) {$R$};
		\node[draw, right= of I12] (I2) {$I_2$};
		\draw[->] (S) -- (I12) ;
		\draw[->,] (I2) -- (I12);
		\draw[->,big arrow] (I1) -- (I12);
		\draw[->,big arrow] (I12) -- (R);

    \draw[->,big arrow] (S) .. controls (2,2) .. (I1) ;
    \draw[->,big arrow] (S) .. controls (4,2) .. (I2) ;

 \draw[->,big arrow] (I2) .. controls (5,-2) .. (R) ;
  \draw[->,big arrow] (I1) .. controls (1,-2) .. (R) ;

		\node[left =of I1] at (2.9,1.5) {$\alpha_1$};
		\node[left =of I1] at (7.2,1.5) {$\alpha_2$};
		\node[left =of I12]at (3.5,0.25) {$\eta_1,\gamma_1$};
		\node[left =of I1] at (5.3,1) {$\alpha_3$};
		\node[left =of I1] at (4.5,-0.9) {$\rho_3$};
		\node[left =of I1] at (6.5,0.25) {$\eta_2,\gamma_2$};
		\node[left =of I2] at (2.6,-1.8) {$\rho_1$};
		\node[left =of I1] at (7.5,-1.8) {$\rho_2$};
		
		\end{tikzpicture}
	\end{center}
	\caption{ Flow diagram for two strains coinfection model.}
	\label{flowdiag}
\end{figure}
}
All the parameters
\begin{equation}\label{Bparam}
{\mathcal B}=(r,\alpha_1,\alpha_2,\alpha_3,\mu_1,\mu_2,\mu_3,\eta_1,\eta_2), \text{ }K\text{ and }\;\;{\mathcal I}=(\gamma_1,\gamma_2,\beta_1,\beta_2)
\end{equation}
 are positive. The reason that we split the parameters into three groups ${\mathcal B}$, $K$ and ${\mathcal I}$ is the following. First that we intend to parameterize branches of equilibrium points by $K$. Second,  the parameters ${\mathcal J}$ are only relevant
when the individuals being infected does not move to the same compartment as the infecter. We will use in what follows a shorter vector notation for the parameters
$$
\alpha=(\alpha_1,\alpha_2,\alpha_3),\;\;\mu=(\mu_1,\mu_2,\mu_3),\;\;\eta=(\eta_1,\eta_2),
\;\;\gamma=(\gamma_1,\gamma_2)\;\;\mbox{and}\;\;\beta=(\beta_1,\beta_2).
$$
The parameters from ${\mathcal B}$ will be called basic and from ${\mathcal I}$ coinfection constants/parameters respectively. The confection constants are involved in terms describing double infection and it is reasonable to assume that they are small with respect to basic constants.
\newline

In addition to the parameters being positive it is also assumed that
\begin{equation}\label{assumption sigma}
    \sigma_1<\sigma_2<\sigma_3
\end{equation}
where
$$\sigma_i=\frac{\mu_i}{\alpha_i},\quad i=1,2,3$$
This corresponds to disease 1 being the disease most inclined to spread through a naive population and coinfection being the hardest to spread. The motivation for the coinfection being least likely to spread is because it requires both diseases to spread simultaneously. Another motivation is that we expect that atleast some of the coinfected have previously been single infected and have already started there process of acquiring immunity. The gradual process of immunisation is currently not captured by the model and probability of acquiring immunity is regarded as a memoryless distribution. As is discussed in \cite{part1} $\sigma_i$, $i=1,2,3$ represent thresholds on $K$ for which the reproductive number $R_0(I_i)=\frac{K}{\sigma_i}$ is equal to one. 

When all parameters in $\Omega$ are zero then it is possible to give a complete description of the dynamics by showing that there is always exactly one globally stable equilibrium point, see \cite{GKTW}. In \cite{part1} and \cite{part2} the model \eqref{SIRmodelrow1}-\eqref{SIRmodelrow4} is studied for the case $\gamma_1,\gamma_2>0$ and $\beta_1,\beta_2=0$ and it is shown that there exist an equilibrium branch parameterized by $K$ that for certain $K$ is a coexistence point. In the case $\gamma_1,\gamma_2>0$ and $\beta_1=\beta_2=0$ there can exist steady equilibrium points with non-zero coinfection compartment while the  compartment for the less infectious disease is zero. In other words, the lesser disease can only exist in a host together with the more infectious disease and not by itself. This scenario is hard to motivate epidemiologically and suggests that the model is incomplete. By introducing $\beta_1,\beta_2$, that represent rates of single infections cause by coinfected individuals, this possibility is removed and we get a more comprehensive model.

In \cite{SKTW18a} model \eqref{SIRmodelrow1}-\eqref{SIRmodelrow4} is thoroughly studied for the case when all the parameters of ${\mathcal I}$ are zero and exact expressions for each equilibrium are found.  When introducing the parameters of ${\mathcal I}$ there will for some choice of parameters exist an inner equilibrium point. The analytical expression of this point however will be too complex to be written out. It is presumed however that the dynamics when the parameters of $\Omega$ are positive will be similar to the dynamics when they are zero if the parameters of ${\mathcal I}$ are small enough. Therefore we will at the moment assume the parameter of ${\mathcal I}$ to be small,
 i.e. they will be estimated by a positive constant depending on the parameters from the first group. In what follows we will denote by $C$ or $c$ possibly with an index different positive constant depending on $\alpha$, $\beta$, $\eta$ and $r$.  Moreover, we use the notation $f=O(\bar{\gamma}+\bar{\beta})$ if $|f|\leq C (\bar{\gamma}+\bar{\beta})$, where
 $$
\bar{\gamma}=\gamma_1+\gamma_2,\;\;\bar{\beta}=\beta_1+\beta_2.
$$
By sufficiently small (large) we mean in what follows that a quantity is estimated by a constant, depending on $\mathcal B$, from above (below) respectively.

\begin{remark}
We note that in \cite{GTKW} the notations of the parameters are somewhat different from what is used in this paper. The parameters $K$ and $r$ in this paper can be expressed with the parameters of $K$, $\mu_0$, $b$ and $S^{**}$ in \cite{GTKW} as
$$
    K=S^{**}=K(1-\frac{\mu_0}{b})
$$
and
\begin{equation*}
    r=b-\mu_0
\end{equation*}
\end{remark}

In \cite{GTKW} equilibrium points are studied and expressions for all equilibrium points together with conditions for existence and stability are expressed. In \cite{GTKW} it is found that there exist a stable equilibrium branch parametrized by $K$. For small $K$ the equilibrium points of the branch are disease free, when $K$ reaches $\sigma_1$ they turn into a single disease point where the single disease is disease 1, which is no surprise since it is the most infectious one.

It is the aim of this paper to show that this branch under certain conditions will bifurcate to a branch of coexistence points. We will also study the stability of the branch.

As it was shown in \cite{SKTW18a}  the system (\ref{SIRmodelrow1})--(\ref{SIRmodelrow4}) with $\beta=0$ and $\gamma=0$ contains a single globally stable equilibrium point which depend on the ${\mathcal B}$-parameters and $K$. We are interesting in dependence of the dynamics on the parameters of the problem. It can be described as follows. For a fixed set of parameters from ${\mathcal B}$ there exists a \textit{continuous} branch of stable equilibrium points parameterized by $K\in (0,\infty)$. The $S$ component of the equilibrium points is separated from zero and to each stable equilibrium point we assign a three dimensional  \textit{binary} vector determining the type of the equilibrium point. For example the vector $(1,0,0)$ corresponds to the type of equilibrium $(S,I_1,I_2,I_{12})$, with $I_2=0$ and $I_{12}=0$. The above branches consist of stable equilibrium points of different types which transform from one type to another through bifurcation points.

It was shown in \cite{SKTW18a} that there are four types of branches (parameterized by $0<K<\infty$) which can be schematically written as follows
\begin{eqnarray}\label{Ju6a}
&&(0,0,0)\rightarrow (1,0,0),\\
\label{Ju6b}
&&(0,0,0)\rightarrow (1,0,0)\rightarrow (1,0,1)\rightarrow (0,0,1),\\
\label{Ju6c}
&&(0,0,0)\rightarrow (1,0,0)\rightarrow (1,0,1)\rightarrow (1,1,1)\rightarrow (0,1,1)\rightarrow (0,0,1),\\
\label{Ju6d}
&&(0,0,0)\rightarrow (1,0,0)\rightarrow (1,0,1)\rightarrow (1,1,1).
\end{eqnarray}
The type of bifurcation branches depends of the choice of the ${\mathcal B}$ parameters and is determined by two (dimensionless) constants
$$
\eta_1^*=\frac{r\eta_1}{\alpha_1\alpha_3(\sigma_3-\sigma_1)},\;\;\;\eta_2^*=\frac{r\eta_2}{\alpha_2\alpha_3(\sigma_3-\sigma_2)}.
$$
More precisely, the types (\ref{Ju6a})--(\ref{Ju6d}) are determined by the inequalities
\begin{equation}\label{Ju6e}
\begin{split}
 \eta_1^*&<1,\\
\eta_2^*&>\eta_1^*>1,\\
\eta_1^*&>\eta_2^*>1\\
\eta_1^*&>1>\eta_2^*\\
\end{split}
\end{equation}
respectively.

In the case of non vanishing $\gamma$ and $\beta$ the dynamics is different and will be considered here. Our \textit{main result} is  the following theorem

\begin{theorem}
Under the made assumptions,  there are only two types of branches. More precisely, the first one is the same as before -- the branch \eqref{Ju6a}, which occurs when $\eta_1^*<1$. The second one appears when $\eta_1^*>1$ and it bifurcates from the branch \eqref{Ju6a} for a certain value of $K$ and after that it consists of coexistence equilibrium points. Now the restrictions \eqref{Ju6e} reflect the following properties of the second (coexistence) branch. When the second inequality in \eqref{Ju6e} is satisfied then the coexistence branch is located near the branch \eqref{Ju6b}, if the third inequality in \eqref{Ju6e} holds then coexistence branch is close to \eqref{Ju6c} and finally if the last inequality in \eqref{Ju6e} is true then the coexistence branch is close to \eqref{Ju6d}. This property implies in particular, that for a fixed large $K$ small changes in the ${\mathcal B}$-constants in \eqref{Bparam} around surfaces $\eta_j^*=1$, $j=1,2$, or $\eta_1^*=\eta_2^*$ can cause large changes in the system dynamics.
\end{theorem}

\subsection{Basic properties}

We will now study the existence and attributes of coexistence equilibrium points of \eqref{SIRmodelrow1}-\eqref{SIRmodelrow4}. Consequently we assume the derivatives to be zero and all compartments to be positive. We get the system of equations
\begin{eqnarray}
&&\Big(r(1-\frac{S}{K})-\alpha_1I_1 -\alpha_2I_2 -(\beta_1 +\beta_2 + \alpha_3)I_{12}\Big)S=0,\label{SIRequilibriumrow1}\\
&& (\alpha_1S-\eta_1I_{12} -\gamma_1I_2 -\mu_1)I_1 +\beta_1SI_{12}=0\label{SIRequilibriumrow2}\\
&&(\alpha_2S-\eta_2I_{12}-\gamma_2I_1 -\mu_2)I_2 +\beta_2SI_{12}=0,\label{SIRequilibriumrow3}\\
&&( \alpha_3S+ \eta_1I_1 +\eta_2I_2 -\mu_3)I_{12} +( \gamma_1 +\gamma_2)I_1I_2=0\label{SIRequilibriumrow4}.
\end{eqnarray}
For compactness we will use the notations
\begin{eqnarray*}
&& A_1=\frac{\alpha_1\alpha_3}{r}(\sigma_3-\sigma_1), \quad \eta_1^*=\frac{\eta_1}{A_1},
\\
&& A_2=\frac{\alpha_2\alpha_3}{r}(\sigma_3-\sigma_2), \quad \eta_2^*=\frac{\eta_2}{A_2},\\
&& A_3=\frac{\alpha_1\alpha_2}{r}(\sigma_2-\sigma_1).
\end{eqnarray*}
Note that all the parameters above are positive under assumption \eqref{assumption sigma}.

We always assume that
\begin{equation}\label{J6c}
\Delta_\alpha:=\alpha_2\eta_1-\alpha_1\eta_2\neq 0
\end{equation}
and put
$$
\Delta_\mu=\mu_2\eta_1-\mu_1\eta_2.
$$
These quantities satisfy the inequalities (see (7) in \cite{part2})
\begin{equation}\label{J27a}
\Delta_\mu>\sigma_1\Delta_\alpha\;\;\;\mbox{and}\;\;\;\Delta_\mu>\sigma_2\Delta_\alpha.
\end{equation}
Since
$$
\eta_1^*-\eta_2^*<\eta_1\frac{\alpha_2}{\alpha_1A_2}-\eta_2^*=\frac{\Delta_\alpha}{\alpha_1A_2}
$$
we have
$$
\Delta_\alpha>0\;\;\;\mbox{if $\eta_1^*>\eta_2^*$}.
$$
Using relations (5) and (8) in \cite{part2}, we get
\begin{equation}\label{J4b}
\sigma_3>\frac{\Delta_\mu}{\Delta_\alpha}>\sigma_2\;\;\;\mbox{if $\eta_1^*>\eta_2^*$}.
\end{equation}
Another useful relation is
\begin{equation}\label{J27b}
\eta_1^*-\eta_2^*=\frac{(\sigma_3\Delta_\alpha-\Delta_\mu)\alpha_3}{rA_1A_2}.
\end{equation}
We will assume in what follows that
\begin{equation}\label{J28a}
\eta_1^*\neq 1,\;\;\eta_2^*\neq 1\;\;\mbox{and $\eta_1^*\neq \eta_2^*$}.
\end{equation}

Let us remind some known results proved in \cite{SKTW18a}. The problem (\ref{SIRequilibriumrow1})--(\ref{SIRequilibriumrow4}) has three boundary equilibrium points (with one of the components being zero):
$$
O=(0,0,0,0),\;\;G_{000}=(K,0,0,0)\;\;\mbox{and}\;\;
$$
and
$$
G_{100}=\Big(\sigma_1,\frac{r}{K\alpha_1}(K-\sigma_1),0,0\Big),\;\;
G_{010}=\Big(\sigma_1,0,\frac{r}{K\alpha_2}(K-\sigma_2),0\Big).
$$
The point $O$ is unstable. The point $G_{000}$ is globally stable for $K\leq\sigma_1$. It is unstable for $K>\sigma_1$.
The point $G_{010}$ is always unstable for small $\gamma$.

Let us consider how the bifurcation point appears on the branch $G_{100}(K)$.
The Jacobian matrix at $G_{100}(K)$ is calculated easily as
$$
J=\begin{bmatrix}\label{JacobianG3}
    -r\frac{\sigma_2}{K} & -\alpha_1\sigma_1 & -\alpha_2 \sigma_1 & -(\alpha_3+\beta_2)\sigma_1
    \\
    \alpha_1 I_1^* & 0 & -\gamma_1 I_1^* & -\eta_1 I_1^*+\beta_1\sigma_1
   \\
   0 & 0 & -\alpha_2(\sigma_2-\sigma_1)-\gamma_2I_1^* & \beta_2 \sigma_1
   \\
   0 & 0 & \gamma I_1^*  & \alpha_3\sigma_1+\eta_1I_1^*-\mu_3.
    \end{bmatrix}
$$
Proposition 6 in \cite{GKTW} states that this equilibrium is stable if and only if
\begin{eqnarray*}
&&\det{H}=\left|\begin{matrix}
-\frac{rA_3}{\alpha_1}-\gamma_2 I_1^* & \beta_2\sigma_1
\\
\gamma I_1^* & -\frac{rA_1}{\alpha_1}+\eta_1I_1^*
\end{matrix}\right|>0
\end{eqnarray*}
The bifurcation point satisfies
\begin{equation}\label{J5a}
\Big(\frac{rA_3}{\alpha_1}+\gamma_2 I_1^* \Big)\Big( -\frac{rA_1}{\alpha_1}+\eta_1I_1^*\Big)+\beta_2\sigma_1\bar{\gamma} I_1^*=0.
\end{equation}
This equation with respect to $I_1^*$ has only one positive root
$$
I_1^*(K)=\frac{r}{\alpha_1\eta_1^*}+O(\beta_2\bar{\gamma}).
$$
It can be resolved with respect to $K$ and we get
\begin{equation}\label{K_0asymptotics}
K=K_0=\frac{\eta^*_1\sigma_1}{\eta_1^*-1}+O(\beta_2\bar{\gamma}).
\end{equation}
The second root of \eqref{J5a} is negative and can be estimated as $\leq -\frac{rA_3}{\alpha_1\gamma_2}$.

All statements of the following proposition can be verified quite straightforward
\begin{proposition}
Assume $\eta^*_1>1$ then there exist positive constants $C_1$ and $C_2$ depending on $\Gamma$ constants such that if $\bar{\beta}+\bar{\gamma}\leq C_1$
then the branch $G_{100}(K)$, $K>\sigma_1$, has a unique bifurcation point $K=K_0$ which has the asymptotics (\ref{K_0asymptotics}).
At $K=K_0$ one of the eigenvalues of the Jacobian matrix for $G_{100}$ is zero while the remaining eigenvalues has negative real part and are estimated by $\leq -C_2$.
The zero eigenvalue has an eigenvector
$$
    \Big(0,0,\beta_2\sigma_1,\; \frac{r}{\alpha_1}A_3+\gamma_2\Big(1-\frac{\sigma_1}{K}\Big)\Big).
$$
\end{proposition}

For non-trivial equilibrium point we have the following estimates
\begin{equation}\label{K11ag}
\sigma_1\leq S\leq \sigma_3
\end{equation}
and
\begin{equation}\label{K11a}
I_1\leq\frac{r}{\alpha_1},\;\;I_2\leq\frac{r}{\alpha_2},\;\;I_{12}\leq\frac{r}{\alpha_3+\beta_1+\beta_2}.
\end{equation}

\subsection{Coexistence equilibrium points near $G_{100}(K_0)$.}

What we know thus far is that there occurs a bifurcation at $K=K_0$. However we still do not know if the resulting equilibrium point is a coexistence point especially that all compartments are positive. Proposition \ref{Pr1a} in section \ref{sectionbufurcation} gives us the asymptotics for the compartments $(y_1,y_2)=(I_2,I_{12})$ that starts of as zero at the bifurcation point. In order to use Proposition \ref{Pr1a} we first need to do the suitable change of variables as in Section~ \ref{section Change of variables}.

\begin{proposition}\label{P1}
Let $\eta_1^*>1$.
 Then there exist $C>0$ such that if $\bar{\beta}+\bar{\gamma}\leq C$ then the equilibrium point $G_{100}(K)=(\sigma_1,\frac{r}{\alpha_1K}(K-\sigma_1),0,0)$ bifurcates into a branch of inner equilibrium points $G^*(K)=(S^*,I_1^*,I_2^*,I_{12}^*$) at the point $K=K_0$ with asymptotics \eqref{K_0asymptotics}.
The inner equilibrium has the asymptotics
\begin{eqnarray*}
&&S^*=\sigma_1+O(|K-K_0|)
\\
&&I_1^*=\frac{r}{\eta_1^*\alpha_1}+O(|K-K_0|)
\\
&&I_2^*=\beta_2\Big(\frac{\sigma_1}{\alpha_2(\sigma_2-\sigma_1)}+O(|K-K_0|+\gamma_2)\Big)I_{12}^*
\\
&& I_{12}^*=\Big(\frac{r^2A_1A_3(\eta_1^*-1)^2}{\eta_1^*\mu_1\alpha_1\omega}+O(\gamma_2)\Big)(K-K_0)+O((K-K_0)^2)
\end{eqnarray*}
where
\begin{equation}\label{Jbc}
\omega=\frac{\alpha_2(\sigma_2-\sigma_1)r\eta_1^2}{\alpha_1^2\sigma_1K_0}+O(\bar{\beta}+\bar{\gamma}).
\end{equation}
Moreover the small eigenvalue of the Jacobian matrix on this curve has the following asymptotics
\begin{equation}\label{J1c}
\lambda_0(K)=-\frac{r\sigma_1}{\alpha_1K_0^2}(K-K_0)+O((K-K_0)^2)
\end{equation}
\end{proposition}

\begin{remark}
We have that
$$
\frac{I_2^*}{I_{12}^*}=\beta_2\Big(\frac{\sigma_1}{\alpha_2(\sigma_2-\sigma_1)}+O(|K-K_0|+\gamma_2)\Big)
$$
 so the ratio $\frac{I_2^*}{I_{12}^*}$ vanishes when $\beta_2\rightarrow 0$. From \cite{part1}, \cite{part2} we have that for $\beta=0$, the point $K=\frac{\eta^*\sigma_1}{\eta_1^*-1}$ is the bifurcation point for which the equilibrium branch studied in \cite{part1} and \cite{part2} goes from being a single infection point $(S^*,I^*,0,0)$ to being a coinfection point with zero disease 2 compartment $(S^*,I_1^*,0,I^*_{12})$. This completely agrees with our assumption that the dynamics of \eqref{SIRmodelrow1}-\eqref{SIRmodelrow4} will be similar for the case when $\beta=0$ and the cases where $\beta$ is sufficient small. However we do get an interesting deviation when $\beta>0$ as the $G_{100}$ equilibrium bifurcates directly into a coexistence point without being of the type $(S^*,I_1^*,0,I_{12})$ first.
\end{remark}
\begin{proof}
In order to apply results of section \ref{sectionbufurcation} we formulate \eqref{SIRequilibriumrow1}-\eqref{SIRequilibriumrow4} in the form \eqref{1}-\eqref{2}, in other words
we have
\begin{align*}
F(x,y;K)&=0
\\
yH(x,y;K)&=0
\end{align*}
where
$$
x_1=S,\;\;x_2=I_1,\;\;y_1=I_2\;\;\mbox{and}\;\;y_2=I_{12}.
$$
with
\begin{eqnarray*}
&&F_1(x,y;K)=\Big(r(1-\frac{x_1}{K})-\alpha_1x_2 -\alpha_1y_1 -(\beta_1 +\beta_2 + \alpha_3)y_2\Big)x_1,
\\
&&F_2(x,y;K)=(\alpha_1x_1-\eta_1y_2 -\gamma_1y_1 -\mu_1)x_2 +\beta_1x_1y_2
\end{eqnarray*}
and
\begin{eqnarray*}
&&H_{11}=\alpha_2x_1-\eta_2y_2-\gamma_2x_2 -\mu_2;\;\;H_{12}=( \gamma_1 +\gamma_2)x_2,
\\
&&H_{21}=\beta_2x_1,\;\;H_{22}=\alpha_3x_1+ \eta_1x_2 +\eta_2y_1 -\mu_3.
\end{eqnarray*}
The matrix $\hat H$ at the bifurcation point
 $$
 x_1^*=\sigma_1,\;\;x_2^*=\frac{r(K_0-\sigma_1)}{\alpha_1K_0},\;\;y_1=y_2=0
 $$
 is given by
\begin{eqnarray*}
\hat H=H(x^*(K_0),0;K_0)=\begin{bmatrix}
-\alpha_2(\sigma_2-\sigma_1)-\gamma_2x_2^* & \gamma x_2^* \\
\beta_2\sigma_1 & \eta_1x_2^*-\alpha_3(\sigma_3-\sigma_1)\\
\end{bmatrix}
\end{eqnarray*}
The vectors $\theta$ and $\widehat{e}$ from section \ref{sectionbufurcation} are given by
$$
\theta=(\beta_2\sigma_1,\alpha_2(\sigma_2-\sigma_1)+\gamma_2x_2^*)
$$
and
$$
\widehat{e}=\frac{1}{N}(\bar{\gamma}x_2^*,\alpha_2(\sigma_2-\sigma_1)+\gamma_2x_2^*),
$$
where
$$
N=\Big(\alpha_2(\sigma_2-\sigma_1)+\gamma_2x_2^*\Big)^2+\beta_2\bar{\gamma}\sigma_1x_2^*.
$$
The element $Q_2(K)=\theta H(K)\widehat{e}^T$ is given by
$$
Q_2(K)=(x_2(K)-x_2^*)\theta M\widehat{e}^T=(\eta_1+O(\bar{\beta}+\bar{\gamma}))(x_2-x_2^*).
$$
where
$$
M=\begin{bmatrix}
-\gamma_2 & \bar{\gamma} \\
0& \eta_1\\
\end{bmatrix},\;\;\;x_2(K)=\frac{r(K-\sigma_1)}{\alpha_1K}.
$$

To complete our preparation to apply the results from section \ref{sectionbufurcation} we need to evaluate the quantity $\omega$ given by (\ref{J1a}). First we note that
\begin{align*}
Q_n&=Q_2=\alpha_3x_1+\eta_1x_2+\eta_2y_1-\mu_3+O(\bar{\beta}+\bar{\gamma}),\\
\theta\cdot\nabla_yF(x^*,o)&=\alpha_2(\sigma_2-\sigma_1)(-\alpha_3\sigma_1,-\eta_1x_2^*)+O(\bar{\beta}+\bar{\gamma}).
\end{align*}
Therefore by (\ref{J1a})
\begin{align*}
\omega&=\frac{(\alpha_3,\eta_1)}{\alpha_1^2\sigma_1x_2^*}\begin{bmatrix}
0 & \alpha_1\sigma_1 \\
-\alpha_1x_2^*& -\frac{r}{K_0}\\
\end{bmatrix}\alpha_2(\sigma_2-\sigma_1)(-\alpha_3\sigma_1,-\eta_1x_2^*)^t\\
&=\frac{\alpha_2(\sigma_2-\sigma_1)r\eta_1^2}{\alpha_1^2\sigma_1K_0}+O(\bar{\beta}+\bar{\gamma}).
\end{align*}
This coincides with (\ref{Jbc}). By \eqref{162}, the smallest eigenvalue of the inner equilibrium point as a function of $K$ close to $K_0$ is
\begin{equation*}
\lambda_0(K)=-\eta_1\partial_Kx_2(K)|_{K=K_0}+O((K-K_0)^2)=-\frac{r\sigma_1}{\alpha_1K_0^2}+O((K-K_0)^2),
\end{equation*}
which proves (\ref{J1c}).
\end{proof}

\section{Coexistence equilibrium points}
Coexistence equilibrium points for the branch appears when $\eta_1^*>1$. The bifurcation takes place for $K=K_0$ at the point $G_{010}$ and the corresponding bifurcation branch exists for some $K>K_0$. Our forthcoming analysis which is based on considering the system \eqref{Au28a1} as a small perturbation of \eqref{Au28ba}, includes several points:

We show  in section \ref{Positivity of the Jacobian} that the Jacobian at the coexistence equilibrium points does not vanish. This allows the construction of global branches consisting of coexistence equilibrium points for all $K>K_0$ starting at $G_{010}(K_0)$. We prove this fact by showing that this branch can be extended for all $K>K_0$ without hitting the boundary and an important role here is played by Lemma \ref{LJ28}. If we compare the model with zero $\alpha$ and $\beta$ with the model with non-zero $\alpha$ and $\beta$ then there are two different branches containing  locally stable coexistence equilibrium points in the first model but there is only one such branch in the second case. The important result is proved in Section \ref{Stability of the inner equilibrium branch} which shows that the branch of coexistence equilibrium points in the case of small, non-zero $\alpha$ and $\beta$ is located close to a branch of locally stable equilibrium points for the model with $\alpha=0$ and $\beta=0$.

We rewrite the system (\ref{SIRequilibriumrow1})-(\ref{SIRequilibriumrow4}) as
\begin{equation}
\label{Au28a1}
\begin{split}
&g_1S=0,\\
&g_2I_1 +\beta_1SI_{12}=0,\\
&g_3I_2 +\beta_2SI_{12}=0,\\
&g_4I_{12} +( \gamma_1 +\gamma_2)I_1I_2=0,
\end{split}
\end{equation}
where $g_j=g_j(S,I_1,I_2,I_{12})$, $j=1,2,3,4$, are given by
\begin{equation}
\label{Au28ba}
\begin{split}
g_1&=r(1-\frac{S}{K})-\alpha_1I_1-\alpha_2I_2-(\alpha_3+\bar{\beta})I_{12},\\
g_2&=\alpha_1S - \eta_1I_{12}-\gamma_1I_2 - \mu_1\\
g_3&=\alpha_2S - \eta_2I_{12}-\gamma_2I_1- \mu_2,\\
g_4&=\alpha_3S+ \eta_1I_1+\eta_2I_2-\mu_3.
\end{split}
\end{equation}

\subsection{The case $\beta=0$ and $\gamma=0$}\label{SS1a}

In this section we remind ourselves of some facts on the solutions of the system \eqref{Au28a1} with $\beta=0$ and $\gamma=0$. The system can be written as
\begin{equation}\label{I2}
\widehat{g}_1S=0,\;\;\widehat{g}_2I_1 =0,\;\;\widehat{g}_3I_2 =0,\;\;\widehat{g}_4I_{12} =0,
\end{equation}
where $\widehat{g}_j$ are defined by formulas \eqref{Au28ba} with $\beta=0$ and $\gamma=0$.

It will be convenient to consider system \eqref{Au28a1} as a perturbation of (\ref{I2}).
All solutions of (\ref{I2}) can be written explicitly. Moreover local stability of the equilibrium points is equivalent to global stability and the explicit condition for this stability is written in \cite{SKTW18a}. The following proposition is borrowed from \cite{SKTW18a}.
Let
\begin{equation}\label{J6a}
K_1=\sigma_1\frac{\eta_1^*}{\eta_1^*-1},\;\;K_2=\sigma_3\frac{\eta_1^*}{\eta_1^*-1},\;\;
K_3=\sigma_2\frac{\eta_2^*}{\eta_2^*-1},\;\;K_4=\sigma_3\frac{\eta_2^*}{\eta_2^*-1}.
\end{equation}
We also introduce the constants
\begin{equation}\label{J6b}
K_5=\frac{\Delta_\mu}{\Delta_\alpha}\frac{\eta_1^*}{\eta_1^*-1},\;\;K_6=\frac{\Delta_\mu}{\Delta_\alpha}\frac{\eta_2^*}{\eta_2^*-1},
\end{equation}
which are positive if $\eta_1^*>\eta_2^*$. By (\ref{J4b})
$$
K_1<K_5<K_2\;\;\mbox{and}\;\;K_3<K_6<K_4\;\;\mbox{if $\eta_1^*>\eta_2^*$}.
$$

\begin{proposition}\label{pro:equil}    Let us consider the equation \eqref{I2}.
There exist only the following non-trivial equilibrium states
\begin{align*}
G_{000}&=(K,0,0,0),\\
G_{100}  &=\left(\sigma_1,\frac{r}{K\alpha_1}(K-\sigma_1), 0,0\right),\\
G_{010}& =\Big(\sigma_2,0,\frac{r}{K\alpha_2}(K-\sigma_2), 0\Big),\\
G_{001} &=\Big(\sigma_3,0,0,\frac{r}{K\alpha_3}(K-\sigma_3)\Big),\\
G_{101}&=\Big(S^*,\frac{\alpha_3}{\eta_1}(\sigma_3-S^*),0,\frac{\alpha_1}{\eta_1}(S^*-\sigma_1)\Big), \quad S^*=K(1-\frac{1}{\eta_1^*}), \\
G_{011}&=(S^*,0,\frac{\alpha_3}{\eta_2}(\sigma_3-S^*),\frac{\alpha_2}{\eta_2}(S^*-\sigma_2)), \quad
S^*=K(1-\frac{1}{\eta_2^*}),\\
G_{111}&=\Big(\frac{\Delta_\mu}{\Delta_\alpha},\frac{rA_2}{\Delta_\alpha}\Big(1-\eta_2^*\Big(1-\frac{\Delta_\mu}{K\Delta_\alpha}\Big)\Big),
\frac{rA_1}{\Delta_\alpha}\Big(\eta_1^*\Big(1-\frac{\Delta_\mu}{K\Delta_\alpha}\Big)-1\Big),\frac{rA_3}{\Delta_\alpha}\Big) ,
\end{align*}
where the equilibrium point $G_{111}$ has positive components when $\eta_1^*>\eta_2^*$.
\end{proposition}
All above equilibrium points depend on $K$ and we will write $G_{ijk}(K)$, $i,j,k=0,1,$ where $i,j,k$ indicates if the respective disease compartment is zero or not. As functions of $K$ all equilibria, except of $G_{111}$, have bifurcation points (where the Jacobian vanishes). It appears that they are one of (\ref{J6a}) and (\ref{J6b}) except of $G_{000}$.

In this section we will consider the case when the branch $G_{100}$ depending on $K$ has a bifurcation, which may happen
at the point $K_3$.

The case of vanishing $\beta$ and $\gamma$ was considered in \cite{SKTW18a}. Let us recall the main result. There are four types of branches defined for $K>0$ and consisting of globally  stable equilibrium points (they are also asymptotically stable outside bifurcation points). 
They correspond to different sets of parameters $\alpha$, $\mu$, $\eta$ and $r$ and are defined as follows:

\begin{enumerate}
  \item [(a)]  The branch $\mathcal{S}_1(K)$ is defined for $\eta_1^*<1$  and consists of $G_{000}(K)$, $0<K<\sigma_1$, and then of $G_{100}(K)$, $K>\sigma_1$. The bifurcation point is $K=\sigma_1$.  We use the diagram
      $$
      G_{000}\rightarrow G_{100}
      $$
      for this branch.

   \item [(b)] The branch $\mathcal{S}_2(K)$ is defined for $\eta_2^*>\eta_1^*>1$ and consists of the branch $G_{000}(K)$ for $0<K<\sigma_1$, $G_{100}(K)$ for $\sigma_1<K<K_1$, $G_{101}(K)$ for $K_1<K<K_2$ and $G_{001}(K)$ for $K_2<K$. There are three bifurcation points $\sigma_1$, $K_1$ and $K_2$, in other words
        $$
        G_{000}\rightarrow G_{100}\rightarrow G_{101}\rightarrow G_{001}.
        $$

  \item [(c)] The third branch $\mathcal{S}_3(K)$ is met when $\eta_1^*>\eta_2^*>1$ and consists of $G_{000}(K)$ for $0<K<\sigma_1$,  $G_{100}(K)$ for $\sigma_1<K<K_1$,  $G_{101}(K)$ for $K_1<K<K_5$, $G_{111}(K)$ for $K_5<K<K_6$, of $G_{011}(K)$ for $K_6<K<K_4$ and of $G_{001}$ for $K_4<K$. There are five bifurcation points here: $\sigma_1,K_1,K_5,K_6$ and $K_4$, in other words
      $$
      G_{000}\rightarrow G_{100}\rightarrow G_{101}\rightarrow G_{111}\rightarrow G_{011}\rightarrow G_{001}.
      $$

  \item [(d)] The last branch $\mathcal{S}_4(K)$ is defined when $\eta_1^*>1>\eta_2^*$ and consists of $G_{000}(K)$ for $0<K<\sigma_1$,  $G_{100}(K)$ for $\sigma_1<K<K_1$,  $G_{101}(K)$ for $K_1<K<K_5$ and $G_{111}(K)$ for $K_5<K$. There are three bifurcation points in this case: $\sigma_1,K_1$ and $K_5$, in other words
 $$
 G_{000}\rightarrow G_{100}\rightarrow G_{101}\rightarrow G_{111}.
 $$

\end{enumerate}



\subsection{A representation for Jacobian}

Let $G=(S,I_1,I_2,I_{12})$ be a coexistence equilibrium point.
Then
the Jacobian matrix $J$ of the vector function in the left hand side of \eqref{Au28a1}  at a coexistence equilibrium point $G$ is
\begin{equation*}
J=\begin{bmatrix}
  -\frac{r}{K}S & -\alpha_1S & -\alpha_2S & -(\beta+\alpha_3)S
\\
\alpha_1I_1+\beta_1I_{12} & -\beta_1S\frac{I_{12}}{I_{1}} &-\gamma_1I_1 & -\eta_1I_1+\beta_1S
\\
\alpha_2I_2+\beta_2I_{12} & -\gamma_2I_{2} & -\beta_2 \frac{SI_{12}}{I_{2}} &-\eta_2I_2+\beta_2S
\\
\alpha_3I_{12} & \eta_1I_{12}+\gamma I_{2}& \eta_2I_{12}+\gamma I_1 & -\gamma\frac{I_1 I_2}{I_{12}}
\end{bmatrix}
\end{equation*}
Then its determinant is
\begin{equation}
   \det{J}=SI_1I_2I_{12}\det{
   \begin{bmatrix}
   -\frac{r}{K} & -\alpha_1 & -\alpha_2 & -\beta-\alpha_3
   \\
   \alpha_1+\beta_1\frac{I_{12}}{I_1} & -\beta_1S\frac{I_{12}}{I_{1}^2} & -\gamma_1 & -\eta_1+\beta_1\frac{S}{I_1}
   \\
   \alpha_2+\beta_2\frac{I_{12}}{I_{2}} & -\gamma_2 &  -\beta_2\frac{SI_{12}}{I_2^2}  & -\eta_2 +\beta_2\frac{S}{I_{2}}
   \\
   \alpha_3 & \eta_1+\gamma\frac{I_{2}}{I_{12}} & \eta_2+\gamma \frac{I_{1}}{I_{12}} & -\gamma \frac{I_1I_2}{I_{12}^2}
   \end{bmatrix}
   }.
\end{equation}
Therefore
\begin{equation}\label{M1}
\frac{1}{SI_1I_2I_{12}}\det{J}= \beta_2\frac{SI_{12}}{I_2^2}A+\beta_1\frac{SI_{12}}{I_1^2}B+C,
\end{equation}
where
\begin{align*}
A=& \det \begin{bmatrix}
   \frac{r}{K} & \alpha_1 & \beta+\alpha_3
   \\
   \alpha_1+\beta_1\frac{I_{12}}{I_1} & -\beta_1S\frac{I_{12}}{I_{1}^2} & -\eta_1+\beta_1\frac{S}{I_1}
   \\
   \alpha_3 & \eta_1+\gamma\frac{I_{2}}{I_{12}} &  -\gamma \frac{I_1I_2}{I_{12}^2}.
   \end{bmatrix}=\beta_1\gamma\frac{rSI_2}{KI_1I_{12}}+\beta_1\alpha_1\alpha_3\frac{S}{I_1}\\
   &+\beta(\alpha_1+\beta_1\frac{I_{12}}{I_1})(\eta_1+\gamma\frac{I_2}{I_{12}})+\alpha_3\beta_1\frac{I_{12}}{I_1}(\eta_1+\gamma\frac{I_2}{I_{12}})+
   \alpha_3\alpha_1\gamma\frac{I_2}{I_{12}}+\beta_1\alpha_3(\beta+\alpha_3)\frac{SI_{12}}{I_1^2}\\
   &+\frac{r}{K}(\eta_1+\gamma\frac{I_2}{I_{12}})(\eta_1-\beta_1\frac{S}{I_1})+\alpha_1\gamma\frac{I_1I_2}{I_{12}^2}(\alpha_1+\beta_1\frac{I_{12}}{I_1})
   =A_1+\frac{r}{K}\eta_1\Big(\eta_1-\beta_1\frac{S}{I_1}+\bar{\gamma}\frac{I_2}{I_{12}}\Big)+\beta\alpha_1\eta_1
\end{align*}
where
\begin{align*}
A_1=&\beta_1\alpha_1\alpha_3\frac{S}{I_1}+\beta\beta_1\frac{I_{12}}{I_1}(\eta_1+\gamma\frac{I_2}{I_{12}})
+\beta\alpha_1\gamma\frac{I_2}{I_{12}}   +\alpha_3\beta_1\frac{I_{12}}{I_1}(\eta_1+\gamma\frac{I_2}{I_{12}})
   \\
 &+\alpha_3\alpha_1\gamma\frac{I_2}{I_{12}} +\beta_1\alpha_3(\beta+\alpha_3)\frac{SI_{12}}{I_1^2}
   +\alpha_1\gamma\frac{I_1I_2}{I_{12}^2}(\alpha_1+\beta_1\frac{I_{12}}{I_1})>0,\\
B=& \det \begin{bmatrix}
   \frac{r}{K} & \alpha_2 & \beta+\alpha_3
   \\
   \alpha_2+\beta_2\frac{I_{12}}{I_{2}} & 0 &   -\eta_2 +\beta_2\frac{S}{I_{2}}
   \\
   \alpha_3 &  \eta_2+\gamma \frac{I_{1}}{I_{12}} & -\gamma \frac{I_1I_2}{I_{12}^2}
   \end{bmatrix}=\beta(\alpha_2+\beta_2\frac{I_{12}}{I_2})(\eta_2+\gamma\frac{I_1}{I_{12}})\nonumber\\
   &+\alpha_3\beta_2\frac{I_{12}}{I_2}(\eta_2+\gamma\frac{I_1}{I_{12}})+\alpha_3\alpha_2\gamma\frac{I_1}{I_{12}}
   +\alpha_2\alpha_3\beta_2\frac{S}{I_2}\\
   &+\gamma\alpha_2\frac{I_1I_2}{I_{12}}(\alpha_2+\beta_2\frac{I_{12}}{I_2})
   +\frac{r}{K}(\eta_2+\gamma\frac{I_1}{I_{12}})(\eta_2-\beta_2\frac{S}{I_2})\nonumber\\
   \end{align*}
and
\begin{equation}\label{M1c}
\begin{split}
C&=\begin{bmatrix}
   -\frac{r}{K} & -\alpha_1 & -\alpha_2 & -\beta-\alpha_3
   \\
   \alpha_1+\beta_1\frac{I_{12}}{I_1} & 0 & -\gamma_1 & -\eta_1+\beta_1\frac{S}{I_1}
   \\
   \alpha_2+\beta_2\frac{I_{12}}{I_{2}} & -\gamma_2 &  0 & -\eta_2 +\beta_2\frac{S}{I_{2}}
   \\
   \alpha_3 & \eta_1+\gamma\frac{I_{2}}{I_{12}} & \eta_2+\gamma \frac{I_{1}}{I_{12}} & -\gamma \frac{I_1I_2}{I_{12}^2}.
   \end{bmatrix}\\
   &=\Delta_\alpha^2+O(\gamma)+O\Big(\frac{\beta_1}{I_1}\Big)+O\Big(\frac{\beta_2}{I_2}\Big),
   \end{split}
\end{equation}
where $\Delta_\alpha$ is defined by (\ref{J6c}).

We note that by (\ref{SIRequilibriumrow2})
 \begin{equation}\label{Koz1}
 I_{12}\Big(\eta_1-\frac{\beta_1S}{I_1}+\gamma_1\frac{I_1}{I_{12}}\Big)=\alpha_1S-\mu_1-\gamma_1I_2
 \end{equation}
and therefore $A$ is positive if $S>\sigma_1+\gamma_1I_2/\alpha_1$. Thus expression (\ref{Koz1}) is positive in a neighborhood of the bifurcation points of the branches  $G_{001}$, $G_{101}$ and $G_{011}$ when $\gamma_1$ is small. Similarly from \eqref{SIRequilibriumrow3} we get
\begin{equation}\label{Koz1a}
 I_{12}(\eta_2-\frac{\beta_2S}{I_2})=\alpha_2S-\mu_2-\gamma_2I_1.
 \end{equation}
The expression is positive if $S>\sigma_2+\gamma_2I_1/\alpha_2$. In particular, if $S>\sigma_2+\delta$, where $\delta$ is a positive number depending on the basic constants, then
\begin{equation}\label{Sept1a}
A\geq \frac{c}{K}\;\;\;\mbox{and}\;\;\; B\geq\frac{c}{K}\;\;\mbox{for small $\beta$ and $\gamma$}.
\end{equation}
Here $c$ is a positive constant depending on the ${\mathcal B}$ parameters  in \eqref{Bparam}.

\subsection{Lemma on properties of coexistence equilibrium points}

One can verify directly that $G_{000}(K)$, $G_{100}(K)$ and $G_{010}(K)$, which are presented in Proposition~\ref{pro:equil}, solve the system \eqref{SIRequilibriumrow1}-\eqref{SIRequilibriumrow4} for all $\beta$ and $\gamma$. All other non-trivial solutions of \eqref{SIRequilibriumrow1}-\eqref{SIRequilibriumrow4} are coexistence equilibrium points and their components satisfy the inequalities (\ref{K11ag}) and (\ref{K11a}). Since the point $G_{000}(K)$ is globally stable we have $K\geq\sigma_1$ for coexistence equilibrium points. We denote the set of all coexistence equilibrium points by
\begin{equation}\label{Mdef}
{\mathcal M}=\{ (G;K), G=(S,I_1,I_2,I_{12}),\;\mbox{satisfies \eqref{SIRequilibriumrow1}-\eqref{SIRequilibriumrow4}}\;:\; \;S,I_1,I_2,I_{12},K>0\}.
\end{equation}
The inequalities (\ref{K11ag}), (\ref{K11a}) and $K\geq\sigma_1$ are valid for points from ${\mathcal M}$.
 In what follows we will use also the notation $G(K)$ together with $(G,K)$. We remind that the assumption \eqref{J28a} is valid everywhere in this paper.

 \begin{lemma}\label{LJ28} Let $(G,K)\in {\mathcal M}$. There exists positive constants $\widehat{\varepsilon}$ and $c$ depending only on ${\mathcal B}$-constants such that if
 \begin{equation}\label{J30b}
 0<\varepsilon\leq\widehat{\varepsilon}\;\;\;\mbox{ and}\;\;\;\bar{\beta}+\bar{\gamma}\leq \varepsilon^2,
 \end{equation}
  then the following assertions are valid:

 {\rm (i)} One of the components $I_1,\,I_2,\,I_{12}$ is greater or equal to $\varepsilon$.

 {\rm (ii)} If $I_2\geq \varepsilon$ then one of $I_1,\,I_{12}$ must be $\geq \varepsilon$.

{\rm (iii)} If $I_1,\,I_2\geq\varepsilon$ then $I_{12}\geq\varepsilon$.

 {\rm (iv)} If $I_1,\,I_2\leq\varepsilon$ and $I_{12}\geq\varepsilon$ then $\eta_1^*>1$, $\eta_2^*>1$ and
 \begin{equation}\label{J30a}
 G(K)=G_{001}(K)+O(\varepsilon),
 \end{equation}
 where
$$
 K\geq K_4-c\varepsilon\;\;\mbox{if $\eta_1^*>\eta_2^*$ and}\;\;K>K_2-c\varepsilon\;\;\mbox{if $\eta_1^*<\eta_2^*$}.
$$

 {\rm (v)} If $I_1\leq\varepsilon$ and $I_2,\,I_{12}\geq\varepsilon$ then $1<\eta_2^*<\eta_1^*$ and
 $$
 G(K)=G_{011}(K)+O(\varepsilon),\;\;\;\mbox{where}\;\;K_4-c\varepsilon\leq K\leq K_6+c\varepsilon.
 $$

 {\rm (vi)} If $I_2\leq\varepsilon$ and $I_1,\,I_{12}\geq\varepsilon$ then $1<\eta_1^*$ and
 $$
 G(K)=G_{101}(K)+O(\varepsilon),
 $$
 where
 $$
 K_1-c\varepsilon \leq K\leq K_2+c\varepsilon\;\;\mbox{if $\eta_2^*>\eta_1^*$ and}\;\;K_1-c\varepsilon \leq K\leq K_5+c\varepsilon\;\;\mbox{if $\eta_1^*>\eta_2^*$}.
 $$

{\rm (vii)} If $I_2,I_{12}\leq\varepsilon$ and $I_1\geq\varepsilon$ then $\eta_1^*>1$ and
$$
G(K)=G_{100}(K)+O(\varepsilon),\;\;\mbox{where $\sigma_1\leq K\leq K_1+c\varepsilon$}.
$$

{\rm (viii)} If $I_1,\,I_2,\,I_{12}\geq\varepsilon$ then $\eta_1^*>1$, $\eta_1^*>\eta_2^*$ (hence  $\Delta_\alpha>0$) and
\begin{equation}\label{J29a}
G(K)=G_{111}(K)+O(\varepsilon),
\end{equation}
where $K_5-c\varepsilon\leq K<\infty$ when $\eta_2^*<1$ and $K_5-c\varepsilon\leq K\leq K_6+c\varepsilon$ when $\eta_2^*>1$.
 \end{lemma}
 \begin{proof} (i)  Assume that   $I_1,I_2,I_{12}\leq\varepsilon$. 
Since $S\geq\sigma_1$ from (\ref{SIRequilibriumrow2}) it follows that $S=\sigma_1+O(\varepsilon)$  otherwise $g_2>0$ and hence $I_1=I_{12}=0$.
Using (\ref{SIRequilibriumrow3}), we get that
$$
(\alpha_2(\sigma_2-\sigma_1)+O(\varepsilon))I_2=\beta_2SI_{12}.
$$
Expressing $I_2$ through $I_{12}$ in (\ref{SIRequilibriumrow4}) with the help of the last relation, we get
$$
(\alpha_3(\sigma_3-\sigma_1)+O(\varepsilon))I_{12}=0,
$$
which implies $I_{12}=0$. Since all components are assumed to be positive we obtain a contradiction which proves (i).

(ii)  Assume that $I_{12},I_1\leq \varepsilon$.
From  (\ref{SIRequilibriumrow3}) it follows
$$
0\leq \alpha_2(\sigma_2-S)+\eta_2I_{12}+\gamma_2I_1\leq \beta_2S
$$
which implies
$$
S=\sigma_2+O(\varepsilon).
$$
Using  (\ref{SIRequilibriumrow2}), we obtain
$$
(\alpha_1(\sigma_2-\sigma_1)+O(\varepsilon))I_1+\beta_1SI_{12}=0,
$$
which gives $I_1=I_{12}=0$ if $\varepsilon$ is sufficiently small. This contradiction proves (ii).

(iii) Assume that $I_{12}\leq\varepsilon$.  From \eqref{SIRequilibriumrow2} and \eqref{SIRequilibriumrow3} it follows that
$$
S=\sigma_1+O(\varepsilon)\;\;\mbox{and}\;\;S=\sigma_2+O(\varepsilon),
$$
which is impossible for small $\varepsilon$.

(iv)
From (\ref{SIRequilibriumrow4}) and  (\ref{SIRequilibriumrow1}) we get
\begin{equation}\label{J21as}
S=\sigma_3+O(\varepsilon),\;\;
I_{12}=\frac{r}{\alpha_3}\Big(1-\frac{\sigma_3}{K}\Big)+O(\varepsilon)
\end{equation}
which implies (\ref{J30a}). Furthermore,
\begin{equation}\label{J21as1}
g_2=-\frac{r\eta_1}{\alpha_3}\Big(1-\frac{1}{\eta_1^*}-\frac{\sigma_3}{K}\Big)
+O(\varepsilon)=-\frac{r\eta_1\sigma_3}{\alpha_3}\Big(\frac{1}{K_2}-\frac{1}{K}\Big)+O(\varepsilon)
\end{equation}
and
\begin{equation}\label{J21as2}
g_3=-\frac{r\eta_2}{\alpha_3}\Big(1-\frac{1}{\eta_2^*}-\frac{\sigma_3}{K}\Big)
+O(\varepsilon)=-\frac{r\eta_2\sigma_3}{\alpha_3}\Big(\frac{1}{K_4}-\frac{1}{K}\Big)+O(\varepsilon),
\end{equation}
Both of these quantities must be negative, otherwise equations \eqref{Au28a1} have no positive solutions. This implies
that $\eta_1^*>1$ and $\eta_2^*>1$ and $K>\max (K_2,K_4)+O(\varepsilon)$. This proves the remaining assertions in (iv).

 (v) From equations (\ref{SIRequilibriumrow1}), (\ref{SIRequilibriumrow3}) and (\ref{SIRequilibriumrow4}) it follows
\begin{eqnarray*}
&&r(1-\frac{S}{K})-\alpha_2I_2 - \alpha_3I_{12}+O(\varepsilon)=0,\\
&&\alpha_2S-\eta_2I_{12} -\mu_2+O(\varepsilon)=0,\\
&&\alpha_3S +\eta_2I_2 -\mu_3+O(\varepsilon)=0.
\end{eqnarray*}
Since the matrix of this system is invertible, we get
$$
G=(S,I_1,I_2,I_{12})=G_{011}+O(\varepsilon).
$$
Positivity of $S^*$, which is given by proposition (\ref{pro:equil}), leads to $\eta_2^*>0$.
The formula of $G_{011}$ implies
$$
g_2=\frac{1}{\eta_2}\Big(\Delta_\mu-\Delta_\alpha S^*\Big)+O(\varepsilon),
$$
which must be negative, otherwise $I_1$ and $I_{12}$ vanish for small $\varepsilon$.
If
$$
\Delta_\mu-\Delta_\alpha S^*>0
$$
then we get a contradiction for small $\varepsilon$. So for existence of such $G$ we must assume that
\begin{equation}\label{J26a}
\Delta_\mu-\Delta_\alpha S^*\leq 0.
\end{equation}
Since $\Delta_\mu>\sigma_2\Delta_\alpha$ by (\ref{J27a}), we obtain from (\ref{J26a}) that
$(\sigma_2-S^*)\Delta_\alpha <0$, which implies $\Delta_\alpha>0$ and $\Delta_\mu>0$ by (\ref{J27a}). Using that $S^*\leq \sigma_3$ we get from (\ref{J27b}) and (\ref{J26a}) that $\eta_1^*\geq\eta_2^*$.

(vi)   Similar to (v), we get
$$
G=(S,I_1,I_2,I_{12})=G_{101}+O(\varepsilon).
$$
Using this relation, we get
$$
g_3=\frac{1}{\eta_1}\Big(S^*\Delta_\alpha-\Delta_\mu\Big)+O(\varepsilon),
$$
where $S^*$ is given in proposition \ref{pro:equil}. Therefore for existence of such $G\in {\mathcal M}$ we must assume that $S^*\Delta_\alpha-\Delta_\mu\leq 0$. Consider two cases. First let $\eta_2^*>\eta_1^*$. Then according to (\ref{J27b}) $\sigma_3\Delta_\alpha-\Delta_\mu<0$. Using that $\sigma_1\Delta_\alpha-\Delta_\mu<0$ by (\ref{J27a}) we get $S^*\Delta_\alpha-\Delta_\mu<0$ for all $S^*\in [\sigma_1,\sigma_3]$. If $\eta_1^*>\eta_2^*$ then by (\ref{J4b}) we see that $S^*\Delta_\alpha-\Delta_\mu\leq 0$ only on the interval $K\in [K_1,K_5)$. This readily yields the desired assertion.

(vii) Using relations (\ref{SIRequilibriumrow1}) and (\ref{SIRequilibriumrow2}) we find
$$
G(K)=G_{100}(K)+O(\varepsilon)
$$
and
$$
g_3=\alpha_2(\sigma_1-\sigma_2)+O(\varepsilon),\;\;g_4=\alpha_3(\sigma_1-\sigma_3)+\eta_1\frac{r(K-\sigma_1)}{K\alpha_1}+O(\varepsilon).
$$
For existence of $G$ with positive components we must require
$$
\alpha_3(\sigma_1-\sigma_3)+\eta_1\frac{r(K-\sigma_1)}{K\alpha_1}\leq 0,
$$
which is equivalent to
$$
K\Big(1-\frac{1}{\eta_1^*}\Big)\leq \sigma_1.
$$
If $\eta_1^*<1$ then
$$
g_4=\frac{r\eta_1}{\alpha_1}\Big(\frac{K-\sigma_1}{K}-\frac{1}{\eta_1^*}\Big)+O(\varepsilon)<0.
$$
From (\ref{SIRequilibriumrow3}) and (\ref{SIRequilibriumrow4}) it follows that
$$
I_2=-\frac{\beta_2S}{g_4}I_{12}\;\;\mbox{and}\;\;I_{12}=\frac{\bar{\gamma}\beta_2SI_1}{g_3g_4}I_{12}.
$$
The second relation implies $I_{12}=0$ for small $\varepsilon$.
This proves (vii).

(viii) From \eqref{Au28a1} we get
$$
g_1=0,\quad g_2,g_3,g_4=O(\varepsilon)
$$
This implies, in particular,
$$
\Delta_\alpha S=\Delta_\mu+O(\varepsilon),\;\;\Delta_\alpha I_{12}=\alpha_1\alpha_2(\sigma_2-\sigma_1)+O(\varepsilon).
$$
This implies $\Delta_\alpha>0$ and $\Delta_\mu>0$. These inequalities together with (\ref{J27b}) implies $\eta_2^*<\eta_1^*$. After this we can find $I_1$ and $I_2$ and get the formula (\ref{J29a}). For positivity of $I_1$ and $I_2$ we must require the inequalities for $K$ in the formulation of (viii).
 \end{proof}

 \begin{corollary}\label{Cor3a}  Let $G(K)\in {\mathcal M}$, where ${\mathcal M}$ is defined by \eqref{Mdef}, and let $\widehat{\varepsilon}$ and $\varepsilon$ be the same as in Lemma~\ref{LJ28} and let $\beta$ and $\gamma$ satisfy \eqref{J30b}. Then

 \begin{itemize}
 \item[(i)] If $\eta_1^*<1$ then there are no coexistence equilibriums $G$.

 \item[(ii)] Let $\eta_2^*>\eta_1^*>1$. Then one of the following alternatives for $G(K)$ is valid:
   \begin{itemize}
 \item[a)]
  $I_2,I_{12}\leq\varepsilon$ and $I_1\geq\varepsilon$;
  \item[b)] $I_1,I_2\leq\varepsilon$ and $I_{12}\geq\varepsilon$, wherein
 \begin{equation}\label{Ju1a}
 G(K)={\mathcal S}_2(K)+O(\varepsilon)\;\;\;\mbox{for $K>K_0$}
 \end{equation}
 \end{itemize}

\item[(iii)] Let $\eta_1^*>\eta_2^*>1$. Then one of the following alternatives for $G(K)$ is valid:
 \begin{itemize}
 \item[a)]
 $I_2,I_{12}\leq\varepsilon$ and $I_1\geq\varepsilon$;
 \item[b)] $I_1, I_2,I_{12}\geq \varepsilon$;
 \item[c)] $I_1,I_{12}\leq\varepsilon$ and $I_2\geq\varepsilon$;
 \item[d)] $I_1,I_2\leq\varepsilon$ and $I_{12}\geq\varepsilon$, wherein
  $$
 G(K)={\mathcal S}_3(K)+O(\varepsilon)\;\;\;\mbox{for $K>K_0$}
 $$
 \end{itemize}

\item[(iv)] Let $\eta_1^*>1>\eta_2^*$. Then one of the following alternatives for $G(K)$ is valid:
 \begin{itemize}
 \item[a)]
 $I_2,I_{12}\leq\varepsilon$ and $I_1\geq\varepsilon$;
 \item[b)]$I_1, I_2,I_{12}\geq \varepsilon$,  wherein
 $$
 G(K)={\mathcal S}_4(K)+O(\varepsilon)\;\;\;\mbox{for $K>K_0$}.
 $$
 \end{itemize}
 \end{itemize}
 \end{corollary}
 \begin{proof} (i) Lemma \ref{LJ28} covers all possible locations of coexistence equilibrium points and in all alternatives $\eta_1^*>1$. This proves this assertion.

 (ii) The inequalities $\eta_2^*>\eta_1^*>1$ are met in Lemma \ref{LJ28} (iv), (vi) and (vii). Let the alternative (iv) be valid. Then the relation (\ref{J30a}) holds with $K>K_2-c\varepsilon$. Since $G_5(K)=G_6(K)$ for $K\in [K_2-c\varepsilon,K_2]$ we obtain (\ref{Ju1a}) for the case under consideration. The remaining two alternatives are considered similarly.

 (iii) The inequalities $\eta_1^*>\eta_2^*>1$ are found in Lemma \ref{LJ28} (iv)--(viii) and they correspond to $G_{001},\,G_{011},\,G_{101},\,G_{100}$ and $G_{111}$ respectively. The same argument as in (ii) gives the proof in this case also.

 (iv)  The inequalities $\eta_1^*>1>\eta_2^*$ occur in Lemma \ref{LJ28} (vi)--(viii) and they correspond to $G_{101}$, $G_{100}$ and $G_{111}$ respectively. The argument from (ii) is applicable here also and we get the required result.

 \end{proof}

 The above corollary motivates the following refinement of \eqref{Mdef}
 \begin{equation}
 {\mathcal M}_0=\{ {\mathcal S}_2 \;\mbox{if $\eta_2^*>\eta_1^*>1$},\; {\mathcal S}_3\;\mbox{if $\eta_1^*>\eta_2^*>1$},\;{\mathcal S}_4\;\mbox{if
 $\eta_1>1>\eta_2^*$}\}.\label{Sept1b}
 \end{equation}

\subsection{Positivity of the Jacobian}\label{Positivity of the Jacobian}

\begin{theorem}\label{Tdet1} 
There exists a positive constant $c_*=c_*(\alpha,\mu,\eta, r)$ such that if $\bar{\beta}+\bar{\gamma}\leq c_*$ then
$$
 \det{J}(S,I_1,I_2,I_{12};K)>0
$$
for all solutions to \eqref{SIRequilibriumrow1}-\eqref{SIRequilibriumrow4} with positive components and $K>0$.
\end{theorem}
\begin{proof}  Since the point $G_{000}$ is globally asymptotically stable for $K\in (0,\sigma_1)$, it is sufficient to prove our assertion for $K\geq\sigma_1$ and for components satisfying (\ref{K11ag}) and (\ref{K11a}).
Next, we note that according to (\ref{M1}) and (\ref{M1c}) the determinant is positive for  $K\geq \sigma_1$ and small $\bar{\beta}+\bar{\gamma}$ if all components of solution $G$ to (\ref{SIRequilibriumrow1})-(\ref{SIRequilibriumrow4}) are bounded from below by a positive constant depending on the basic parameters.

So it is sufficient to prove positivity of the Jacobian near the boundary ${\mathcal M}_0$ only. We consider two cases for $K$.

{\bf Large $K$.} Let us assume that
\begin{equation}\label{Klarge}
K\geq \max(K_2,K_4)+c\varepsilon\qquad \text{ and }\qquad 0<\varepsilon\leq\hat{\varepsilon},
\end{equation}
where $\hat{\varepsilon}$ is the same as in Lemma \ref{LJ28}.
According to Lemma \ref{LJ28} large values of $K$ for elements in ${\mathcal M}$ can be found only in the cases (iv) and (viii) when $\eta_1^*>1>\eta_2^*$.
In the case (iv) by (\ref{J21a}) we have
\begin{equation}\label{J21a}
J=\begin{bmatrix}
  -\frac{r}{K}S & -\alpha_1S & -\alpha_2S & -(\beta+\alpha_3)S
\\
O(\varepsilon)&g_2 &-\gamma_1I_1 & O(\varepsilon)
\\
O(\varepsilon) & -\gamma_2I_{2} & g_3 &O(\varepsilon)
\\
\alpha_3I_{12} & \eta_1I_{12}+\gamma I_{2}& \eta_2I_{12}+\gamma I_1 & -\gamma\frac{I_1 I_2}{I_{12}}
\end{bmatrix}
\end{equation}
Therefore
$$
\det J=g_2g_3\alpha_3^2SI_{12}+O(\varepsilon),
$$
where $S$ and $I_{12}$ are given by (\ref{J21as}) and $g_2$ and $g_3$ are given by (\ref{J21as1}) and (\ref{J21as2}) respectively. Therefore the Jacobian is positive for coexistence equilibrium points when \eqref{Klarge} is fulfilled. In the case (viii) according to (\ref{M1})
$$
\det J>SI_1I_2I_{12}(\Delta_\alpha^2+O(\varepsilon))
$$
for the large values of $K$. This implies the positivity of the Jacobian for small $\varepsilon$.

{\bf Bounded $K$.} Now, it remains to prove Theorem \ref{Tdet1} in an $\varepsilon$-neighborhood of the set
$$
{\mathcal M}_{K_*}=\{(G,K)\in {\mathcal M}_0\,:\, K\leq K_*\}.
$$
where $M_0$ is introduced by (\ref{Sept1b}) and $K_*$ is a certain positive number depending on the basic parameters. We denote such neighborhood by ${\mathcal M}_{\varepsilon}$, namely
 $$
 {\mathcal M}_\varepsilon=\{(G,K)\in {\mathcal M}: |G(K)-{\mathcal S}_j(K)|\leq c\varepsilon\,:\,\mbox{for certain}\, j=2,3,4,\;K\in[\sigma_1,K_*]\}.
 $$
In order to prove the positivity of the Jacobian for such equilibrium points it is sufficient to show that every point in $\mathcal M_{K^*} $
 has a neighborhood $U$ such that the Jacobian of all solutions $G$ to \eqref{SIRequilibriumrow1}-\eqref{SIRequilibriumrow4} in $U\cap {\mathcal M}$ is positive.
This assertion is trivial for $(G,K)\in \mathcal{M}_\varepsilon$ close to boundary equilibrium points. Since the considerations in Lemma \ref{LJ28} (i), (ii) are applicable for all $K\geq \sigma_1$ existence of such neighborhoods for $G_{000}$ and $G_{010}$ follows from  Lemma \ref{LJ28} (i), (ii). Furthermore the assertion for equilibrium points which are not bifurcation points can be verified due to the fact that the Jacobian matrix is invertible there.
 Therefore it is sufficient to prove existence of such neighborhoods for bifurcation points on the branches $G_{100}$, $G_{001}$, $G_{101}$ and $G_{011}$ only. Moreover we can use that the $S$-components of the bifurcation points of branches $G_{001}$, $G_{101}$ and $G_{011}$ is greater than $\sigma_2$.

 Consider first the bifurcation points of the branches  $G_{001}$, $G_{101}$ and $G_{011}$. The $S$ components of these points are greater than $\sigma_2$. According to (\ref{Sept1a}) and (\ref{M1})--(\ref{M1c})
\begin{equation}\label{M1x}
\det{J}\frac{1}{SI_1I_2I_{12}}\geq c_1\Big(\beta_2\frac{SI_{12}}{I_2^2}+\beta_1\frac{SI_{12}}{I_1^2}\Big)+\Delta^2
+O(\gamma)+O\Big(\frac{\beta_1}{I_1}\Big)+O\Big(\frac{\beta_2}{I_2}\Big),
\end{equation}
which implies positivity of the Jacobian in 
 a small neighborhood of a point on this branch provided $K-\sigma_3$ is greater then a certain constant.
The case of $G_{001}$ is trivial and the cases when the points are on the branches $G_{101}$ and $G_{011}$ are considered similarly to the above.
We start from constructing such neighbourhoods near bifurcation points of the branches and consider the bifurcation point of the branch $G_{100}(K)$ defined for $K>\sigma_1$.
The bifurcation point is of the form $K_*=K_1+O(\beta_2\bar{\gamma})$.
The Jacobian matrix at this point has three eigenvalues with negative real part (uniformly with respect to small parameters $\beta$ and $\gamma$) and according to Proposition \ref{PK1} the forth eigenvalue has an asymptotic
$$
\lambda_4(K)=-\omega (K-K_3)+O((K-K_3)^2)\;\;\mbox{for $|K-K_3|\leq\varepsilon$},
$$
where $\varepsilon$ is a small positive number depending on the main parameters $\alpha$, $\mu$, $\eta$ and $r$.
This implies that the Jacobian is non-zero for solutions of \eqref{SIRequilibriumrow1}-\eqref{SIRequilibriumrow4}, in a small neighborhood of $G_{100}$. This completes the proof.
\end{proof}

\section{Stability of the inner equilibrium branch}\label{Stability of the inner equilibrium branch}

According to Theorem \ref{Tdet1} there exists a positive constant $c_*$ depending on ${\mathcal B}$-constants  in \eqref{Bparam}  such that
\begin{equation}\label{J2a}
\det J(G)>0\;\;\;\mbox{for all coexistence solutions $G$ of \eqref{SIRequilibriumrow1}-\eqref{SIRequilibriumrow4} satisfying $\bar{\beta}+\bar{\gamma}\leq c_*$ }.
\end{equation}

\begin{theorem}\label{TJu9}
Let us assume that the assumption \eqref{J2a} is satisfied. Then
\begin{itemize}
\item[(i)] If $\eta_1^*<1$ then there are no coexistence equilibrium point.

\item[(ii)] If $\eta_1^*>1$ then there is a smooth branch $G_*(K)$, $K>K_0$, of coexistence equilibrium points having a limit as $K\to K_0$ with $G_*(K_0)=G_{010}(K_0)$. Moreover, all coexistence equilibrium points belong to this branch.
    \end{itemize}

\end{theorem}
\begin{proof} (i) According to Proposition \ref{P1} there exists a branch $G_*(K)$ starting from the point $G_{010}(K_1)$, and defined in an $\varepsilon$-neighborhood of $K_1$. Moreover the part corresponding to $K\in (K_1,K_1+\varepsilon)$ consists of interior points. Due to assumption (\ref{J2a}) this curve can be continued to a maximum interval $(K_1,M)$. Due to inequalities (\ref{K11ag}) and (\ref{K11a}) all limit points obtained as
$$
\widehat{G}=\lim_{j\rightarrow\infty}G_*(K^{(j)})\;\;\mbox{for a certain sequence $\{K^{(j)}\}$, $K^{(j)}\to M$ as $j\to \infty$}
$$
belong to the boundary (have a zero component). Certainly the point $\widehat{G}$ lies on one of curves $G_{000}$ or $G_{010}$. If the Jacobian at $\widehat{G}$ does not vanish then the curve $G_*$ coincides with $G_{000}$ (or $G_{100}$) in a neighborhood of $\widehat{G}$ which is impossible by construction of $G_*$. Therefore the Jacobian at $\widehat{G}$ vanishes. By Theorem \ref{Tdet1} (see the step 1 in the proof)the point $\widehat{G}$ cannot lie on $G_{000}$ branch. If $\widehat{G}$ belong to $G_{010}$ branch then it must coincides with $G_{100}(K_0)$ which is also impossible. This contradiction shows that $M=\infty$.

In order to prove that there are no other inner equilibrium points we may assume that there is a $G$ which solves the problem \eqref{SIRequilibriumrow1}-\eqref{SIRequilibriumrow4} for $K=K^*$. We may assume that $K^*\geq\sigma_1$. Since the Jacobian does not vanish $G$ is situated on a branch of inner equilibrium point and it can be continued for $K<K^*$ until this curve meets the boundary. But by the same argument as above we show that this meeting can occur only at the point $G_{100}(K_0)$. This implies that this entire curve is a part of the curve $G_*$.

(ii) The same argument as in (i) proves that if there is an interior point that it must lie on a equilibrium curve starting from a bifurcation point on the branch $G_{100}(K)$. Since $\eta_1^*<1$ then for small ${\mathcal I}$ constant there are no bifurcation point on $G_{100}$. This proves (ii).

\end{proof}

\begin{theorem} (i) Let $\eta_1^*<1$. Then all locally stable equilibrium points are located on the branch ${\mathcal S}_1$.

(ii) Let
 $\eta_1^*>1$ and $\eta_2^*>1$. Then all coexistence equilibrium points $G_*(K)$ are locally stable for $K\in (K_0,\infty)$ provided $\bar{\beta}+\bar{\gamma}$ is sufficiently small.

(iii) Let
 $\eta_1^*>1>\eta_2^*$. Then for every $K^*>K_0$ there exists a positive constant $c_1$ depending on $\alpha$, $\mu$, $\eta$, $r$ and $K^*$ such that if $\bar{\beta}+\bar{\gamma}\leq c_1$ then coexistence equilibrium points $G(K)$ are locally stable for $K\in (K_0,K^*]$.

 \end{theorem}

\begin{proof} (i) According to Corollary \ref{Cor3a} (i) there are no coexistence equilibrium points in this case and the only branch of locally stable equilibrium points is ${\mathcal S}_1$.

(ii) According to  Corollary \ref{Cor3a} (ii), (iii) coexistence equilibrium points are located near ${\mathcal S}_2$ and ${\mathcal S}_3$. The equilibrium points on these branches are locally  stable  if they are not bifurcation points. In the case of bifurcation points the Jacobian matrix has one zero eigenvalue and three eigenvalues have negative real part. Therefore the Jacobian matrix for the branch $G_*(K)$ has also at least three eigenvalues with negative real part. Using that the Jacobian is positive for coexistence euilibrium points we conclude that the forth eigenvalue must also have negative real part,

(iii) According to  Corollary \ref{Cor3a} (iv) coexistence equilibrium points are located near ${\mathcal S}_4$. For $K\leq K_*$ this branch consists of stable equilibrium points if they are not bifurcation points and the same argument as in (ii) give the proof in this case also.
\end{proof}

\section{Geometry of branches of coexistence equilibrium points}


\subsection{Bifurcation branches of the general problem}

\begin{theorem} There exists positive $\widehat{\varepsilon}$ and  $c$ depending on $\alpha$, $\mu$, $\eta$  such that if $\varepsilon\leq\widehat{\varepsilon}$ and $\bar{\beta}+\bar{\gamma}\leq \varepsilon^2$ then
$$
|G_*(K)-S_j(K)|\leq c\varepsilon.
$$
 The index $j=1,2,3,4$, is chosen according to the inequalities for $\eta_1^*$ and $\eta_2^*$ described after Proposition~\ref{pro:equil} in \rm{(a)--(d)}.

\end{theorem}
\begin{proof} The proof of this assertion follows directly from Corollary \ref{Cor3a} and Theorem \ref{TJu9}.

\end{proof}

\section{Some auxiliary bifurcation results}\label{sectionbufurcation}

\subsection{Interior equilibrium point} \label{section Interior equilibrium point}

Let  $x=(x_1,\ldots,x_m)\in\Bbb R^m$ and $y=(y_1,\ldots,y_{n})\in\Bbb R^n$ be row vectors and let $s\in \Bbb R$ be  a parameter.
Consider the problem
\begin{equation}\label{1}
F(x,y;s)=0
\end{equation}
and
\begin{equation}\label{2}
yH(x,y;s)=0
\end{equation}
where $F=(F_1,\ldots,F_m)^T$ is a column vector and  $H=\{H_{ij}\}_{i,j=1}^n$ is an $n\times n$ matrix. We assume that all elements of $F_k$ and $H_{ij}$ are real valued $C^2$ functions
with respect to all variables.

It is assumed that 
$$
F(x^*,0;0)=0\;\;\mbox{for a certain $x^*\in \Bbb R^m$}
$$
and  the matrix
\begin{equation}\label{4}
\widehat{A}=\{\widehat{A}_{jk}\}_{j,k=1}^{n}=\{\partial_{x_j}F_k(x^*,0;0)\}_{j,k=1}^{n}\;\;\mbox{is invertible.}
\end{equation}
 Then by the implicit function theorem  the equation
\begin{equation}\label{5}
F(\widehat{x}(s),0;s)=0
\end{equation}
has solutions $\widehat{x}(s)$ for $s\in [-\epsilon,\epsilon]$  such that $\widehat{x}(0)=x^*$ and $\widehat{x}(\cdot)\in C^2([-\epsilon,\epsilon])$. Here $\epsilon$ is a  positive number. Thus the vector-function $(\widehat{x}(s),0)$ delivers a solution to problem (\ref{1}) on the interval $[-\epsilon,\epsilon]$.

\begin{remark}\label{remarklocalyunique} This solution is unique in the set
$$
N_\epsilon=\{(x,s):|x-x^*|\leq\epsilon, |s|\leq\sqrt{\epsilon}\}
$$
if $\epsilon$ satisfies (see \cite{part2}, Appendix)
$$
||\widehat{A}^{-1}||\,M_1\sqrt{\epsilon}\leq c_m,\;\;M_1:=\max_{N_\epsilon}\sum_{k+|\alpha|\leq 2}|\partial_x^\alpha\partial_s^kF(x,0;s)|.
$$
Here $c_m$ is a positive constant depending only on $m$. Moreover the  $C^2([-\epsilon,\epsilon])$-norm of the solution can be estimated by a polynomial of the quantities $||\widehat{A}^{-1}||$ and $M_1$.
\end{remark}

For a bifurcation to occur at $(x^*,0,0)$ it is necessary that the matrix
\begin{equation}\label{6}
\widehat{H}=H(x^*,0;0)\;\;\mbox{has a simple eigenvalue $0$.}
\end{equation}
This, in particular, implies existence of  row vectors $\widehat{e}$ and $\theta$ satisfying
\begin{equation}\label{7}
\widehat{H}\widehat{e}^T=0,\;\;\theta\widehat{H}=0\;\;\mbox{and}\;\;\theta \widehat{e}^T=1.
\end{equation}

We want to apply a similar bifurcation result for $n=1$ from \cite{part2}, Appendix. To do this we make a suitable change of variables in the next section.

\subsection{Change of variables}\label{section Change of variables}

Due to condition (\ref{6}) there exist a non-singular $n\times n$-matrix $S$ such that the matrix
\begin{equation}\label{8}
\Lambda=S^{-1}\widehat{H}S\;\;\mbox{has zeros in the last column and the last row.}
\end{equation}
 One can construct such matrix in the following way, for example. We write $S$ as $S=(S_1,\ldots,S_n)$, where $S_j$ is a column vector.
We assume that (i) $S_n=\widehat{e}^T$, (ii) $\theta S_j=0$ for $j=1,\ldots,n-1$ and (iii) the vectors $S_1,\ldots,S_{n-1}$ are linear independent. Let $e_n=(0,\ldots,0,1)$. By (ii) and the last relation in (\ref{7}) we have
$e_n=\theta S,$
which implies $e_n\Lambda=\theta HS=0$ due to the second relation in (\ref{7}). This proves that the last row in $\Lambda$ is zero. Since $Se_n^T=\widehat{e}^T$ we get from the first relation in (\ref{7}) that $\Lambda e_n^T=S^{-1}H\widehat{e}^T=0$, which proves that the last column in $\Lambda$ is zero. Thus (\ref{8}) is proved. It is useful to keep in mind the following relations
\begin{equation}\label{Apr10a}
\Lambda e_n^T=0,\;\;e_n\Lambda=0,\;\;Se_n^T=\widehat{e}^T\;\;\mbox{and}\;\;e_n=\theta S.
\end{equation}

Furthermore, we put
$$
Q(x,y;s)=S^{-1}H(x,y;s)S
$$
and introduce a new variable $z=(z_1,\ldots,z_n)$ by
$$
z_j=yS_j,\;j=1,\ldots,n,\;\;\mbox{or}\;\;z=yS.
$$
Then the system (\ref{1}), (\ref{2}) takes the form
\begin{equation}\label{2a}
F(x,zS^{-1};s)=0,\;\;\;zQ(x,zS^{-1};s)=0.
\end{equation}

We write $z$ as $(\widetilde{z},z_n)$ and corresponding  block form of $Q$ as
\begin{equation}\label{9}
Q=\begin{bmatrix}
\widetilde{Q}&Q_c\\
Q_r&Q_{n}
\end{bmatrix},
\end{equation}
where $\widetilde{Q}$ is $(n-1)\times (n-1)$ matrix, $Q_c$ is a column vector, $Q_r$ is a row vector of length $n-1$ and $Q_n$ is a scalar. Clearly,
\begin{equation}\label{9aa}
Q_c(x^*,0;0)=0,\;\;\;Q_r(x^*,0;0)=0,\;\;\;Q_{n}(x^*,0;0)=0
\end{equation}
and
$$
\begin{bmatrix}
\widetilde{Q}(x^*,0;0)&0\\
0&0
\end{bmatrix}=S^{-1}\widehat{H}S.
$$
Therefore the matrix
$$
\widehat{Q}:=\widetilde{Q}(x^*,0;0)\;\;\mbox{is invertible}
$$
due to the assumption (\ref{6}).
Solving the second equation in (\ref{2a}), we have
$$
\widetilde{z}=-z_nQ_r\widetilde{Q}^{-1}
$$
and
\begin{equation}\label{11}
z_n(Q_{n}-Q_r\widetilde{Q}^{-1}Q_c)=0.
\end{equation}
Since we are looking for a solution for which $z_n$ is not identically zero, we can write instead of (\ref{11})
\begin{equation}
Q_{n}-Q_r\widetilde{Q}^{-1}Q_c=0.
\end{equation}
So the final system for the second solution is
\begin{equation}\label{1x}
F(x,zS^{-1};s)=0,
\end{equation}
\begin{equation}\label{1xy}
\widetilde{z}\widetilde{Q}+z_nQ_c=0
\end{equation}
and
\begin{equation}\label{11x}
Q_{n}-Q_c\widetilde{Q}^{-1}Q_r=0.
\end{equation}
Since $(x,z)=(x^*,0)$  solves this system for $s=0$,  we are planning to apply the implicit function theorem to the system (\ref{1x})-(\ref{11x}).
From the definition of $Q$ it follows
\begin{equation}\label{12}
Q_{n}=e_nQe_n^T=e_nS^{-1}HSe_n^T=\theta H\widehat{e}^T.
\end{equation}
Furthermore, the Jacobian matrix of the left-hand side of \eqref{1x}-\eqref{11x} at the point $(x,z;s)=(x^*,0;0)$ is equal to
\begin{equation}\label{12a}
{\mathcal J}=\begin{bmatrix}
\widehat{A}&D_{\widetilde{z}}F&\partial_{z_n}F\\
0&\widetilde{Q}&0\\
\nabla_xQ_{n}&\nabla_{\widetilde{z}}Q_{n}&\partial_{z_n}Q_{n}
\end{bmatrix},
\end{equation}
where we have used that $Q_c$ and $Q_r$ equals zero at $(x,z;s)=(x^*,0;0)$. So the matrix (\ref{12a}) is invertible if the matrixes  $\widetilde{Q}$ and
$$
\mathcal{A}=\begin{bmatrix}
\widehat{A}&\partial_{z_n}F\\
\nabla_xQ_{n}&\partial_{z_n}Q_{n}
\end{bmatrix}
$$
 are invertible. This is true if the matrix $\widehat{A}$ is invertible and
 \begin{equation}\label{13a}
\omega:=-\frac{\det(\mathcal{A})}{\det(\widehat A)}=\Big(\nabla_xQ_{n}\widehat{A}^{-1}\partial_{z_n}F-\partial_{z_n}Q_{n}\Big)\Big|_{(x,y;s)=(x^*,0;0)}\neq 0.
 \end{equation}
Since $yS=z$ we have that
$$\frac{\partial y}{\partial z_n}S=e_n=\theta S$$
 so $\theta=\frac{\partial y}{\partial z_n}$ and  $\partial_{z_n}=\theta\cdot\nabla_y$.
 Using that $\partial_{z_n}=\theta\cdot\nabla_y$ together with (\ref{12}), we get the following expression for $\omega$ in $(x,y)$ variables
 \begin{equation}\label{J1a}
 \omega=\nabla_x\theta H(x^*,0;0)\widehat{e}^T\widehat{A}^{-1}\theta\cdot\nabla_yF(x^*,0;0)-\theta\cdot\nabla_y\theta H(x^*,0;0)\widehat{e}^T.
 \end{equation}

\begin{proposition}\label{Pr1a} Let the assumptions (\ref{4}), (\ref{6}) and (\ref{13a}) be valid. Then the
system (\ref{1x})-(\ref{11x}) is uniquely solvable in a neighborhood of $(x^*,0;0)$.
This solution $x(s),y(s)$  belongs to $C^2([-\sigma,\sigma])$ for a certain positive $\sigma$. Moreover
\begin{equation}\label{15a}
y(s)=\frac{\partial_sQ_{n}(\widehat{x}(s),0;s)|_{s=0}}{\omega}s\theta +O(s^2).
\end{equation}
\end{proposition}

\begin{proof}
Let us evaluate the derivative $\partial_s(x(s),y(s))(s)$ at $s=0$. We start by evaluating the derivative $\partial_s(x,z)(s)$ at $s=0$.
From (\ref{1xy}) it follows that $\partial_s\widetilde{z}(0)=0$.  Differentiating (\ref{1x}) and (\ref{11x}) with respect to $s$ and using that $\frac{\partial y}{\partial z_n}=\theta$, we get
\begin{equation}\label{14}
\partial_{x_k}F(x^*,0;0)\dot{x}_k+\partial_{y_j}F(x^*,0;0)\dot{z}_n\theta_j+\partial_sF(x^*,0;0)=0
\end{equation}
and
\begin{equation}\label{14a}
\partial_{x_k}Q_{n}(x^*,0;0)(\dot{x}_k-\dot{\widehat{x}}_k)+\partial_{y_j}Q_{n}(x^*,0;0)\dot{z}_n\theta_j+\partial_sQ_{n}(\widehat{x}(s),0;s)|_{s=0}=0
\end{equation}
Here to obtain the first and the last terms in the left-hand side we used the representation $Q_n(x,y;s)=Q_n(x,y;s)-Q_n(\widehat{x}(s),0;s)+Q_n(\widehat{x}(s),0;s)$ and similarly we used that $F(x,y;s)=F(x,y;s)-F(\widehat{x}(s),0;s)$ in (\ref{14}) due to (\ref{5}). The relations (\ref{14}) and (\ref{14a}) implies
$$
-\partial_{x_k}Q_{n}(x^*,0;0)A^{-1}\partial_{y_j}Q_{n}(x^*,0;0)\dot{z}_n\theta_j+\partial_sQ_{n}(\widehat{x}(s),0;s)|_{s=0}=0,
$$
which gives
\begin{equation}\label{15}
z_n=\frac{\partial_sQ_{n}(\widehat{x}(s),0;s)|_{s=0}}{\omega}s+O(s^2)
\end{equation}
Since $(0,\ldots,0,z_n)S^{-1}=z_n\theta$ we arrive at (\ref{15a}).
\end{proof}

\begin{remark} The solution in Proposition (\ref{Pr1a}) is unique in a neighborhood
$$
{\mathcal N}_\sigma=\{(x,y;s):|s|\leq \sigma,\;\; |x-x^*|\leq\sqrt{\sigma},\;\;|y|\leq\sqrt{\sigma}\}.
$$
where $\sigma$ is a positive number depending on
$$
||\widehat{A}^{-1}||,\;\;||\widehat{Q}^{-1}||,\;\; |\omega|^{-1}\;\;\mbox{ and}\; \max_{{\mathcal N}_\sigma}|\partial_x^\alpha\partial_y^\beta\partial_s^k{\mathcal F}|,\;\mbox{ where} \;|\alpha|+|\beta|+k\leq 2.
$$
The $C^2$-norm depends also on the above quantities.
\end{remark}

\subsection{Jacobian matrix}

Consider the vector functions
\begin{equation*}
{\mathcal F}(x,y;s)=\begin{bmatrix}
F(x,y;s)\\
yH(x,y;s)
\end{bmatrix}
\end{equation*}
and
\begin{equation*}
{\mathcal G}(x,z;s)=\begin{bmatrix}
F(x,zS^{-1};s)\\
zQ(x,zS^{-1};s)
\end{bmatrix}.
\end{equation*}
Denote the corresponding Jacobian matrixes by ${\mathcal J}_{{\mathcal F}}$ and ${\mathcal J}_{{\mathcal G}}$ respectively. The first one is taken with respect to $(x,y)$ and the second one with respect to $(x,z)$. One can check by using the definitions of the vector functions and the change of variables that
\begin{equation*}
{\mathcal J}_{{\mathcal F}}(x,y;s)=\begin{bmatrix}
I&0\\
0&\widehat{A}
\end{bmatrix}{\mathcal J}_{{\mathcal G}}(x,z;s)\begin{bmatrix}
I&0\\
0&\widehat{A}^{-1}
\end{bmatrix}.
\end{equation*}
This implies that the eigenvalues of the jacobian matrixes ${\mathcal J}_{{\mathcal F}}$ and ${\mathcal J}_{{\mathcal G}}$ coincide.

\subsection{Stability analysis}

We have
\begin{equation*}
{\mathcal J}_{{\mathcal G}}(x,z;s)=\begin{bmatrix}
D_xF&D_{\widetilde{z}}F&D_{z_n}F\\
\widetilde{z}D_x\widetilde{Q}+z_nD_xQ_c & D_{\widetilde{z}}(\widetilde{z}\widetilde{Q}+z_nQ_c)&D_{z_n}(\widetilde{z}\widetilde{Q}+z_nQ_c))\\
D_x(\widetilde{z}Q_r+z_nQ_{n}) & D_{\widetilde{z}}(\widetilde{z}Q_r+z_nQ_{n}) & D_{z_n}(\widetilde{z}Q_r+z_nQ_{n})
\end{bmatrix}.
\end{equation*}
In particular
\begin{equation*}
{\mathcal J}_{{\mathcal G}}(x^*,0;0)=\begin{bmatrix}
\widehat{A}&D_{\widetilde{z}}F&D_{z_n}F\\
0 & \widetilde{Q}&0\\
0 & 0& 0
\end{bmatrix}.
\end{equation*}
By our assumption both matrixes $\widehat{A}$ and $\widetilde{Q}$ are invertible and so the matrix ${\mathcal J}_{{\mathcal G}}(x^*,0;0)$ has a simple eigenvalue $0$ and all other eigenvalues are non zeroes. This implies that the matrix ${\mathcal J}_{{\mathcal G}}(x(s),z(s);s)$ has the same eigenvalues up to terms of order $O(s)$ and one simple small eigenvalue of order $O(s)$. The last eigenvalue will be denoted by $\lambda(s)$.

To prove our main stability result we will need the following assertion, which can be verified straightworward
\begin{lemma}\label{L13a} Consider $N\times N$ matrix written in a block form
$$
{\mathcal R}=\begin{bmatrix}
{\mathcal R}_{11}&{\mathcal R}_{12}\\
{\mathcal R}_{21} & {\mathcal R}_{22}\\
\end{bmatrix}
$$
where ${\mathcal R}_{11}$ is an invertible $(N-1)\times (N-1)$ matrix, ${\mathcal R}_{12}$ and ${\mathcal R}_{21}$ are $1\times (N-1)$ and $(N-1)\times 1$ matrixes and ${\mathcal R}_{22}$ is a scalar. Then
$$
\det{\mathcal R}=({\mathcal R}_{22}-{\mathcal R}_{21}{\mathcal R}_{11}^{-1}{\mathcal R}_{12})\det{\mathcal R}_{11}.
$$
\end{lemma}

\begin{proposition}\label{PK1} The following formula is valid
\begin{equation}\label{162}
\lambda(s)=-\partial_sQ_n(\widehat{x}(s),0;s)|_{s=0}s+O(s^2).
\end{equation}
\end{proposition}
\begin{proof} First let us show that
\begin{equation}\label{16}
\lambda(s)=\frac{\det {\mathcal J}_{{\mathcal G}}(x(s),z(s);s)}{\det\widehat{A}\det\widehat{Q}}+O(s^2).
\end{equation}

Denote by $\lambda_j(s)$, $j=1,\ldots,m+n,$ the eigenvalues of ${\mathcal J}_{{\mathcal G}}(x(s),z(s);s)$,
where there multiplicities are taken into account. We assume that $\lambda_{m+n}=\lambda$. Since the matrix ${\mathcal J}_{{\mathcal G}}(x(s),z(s);s)$ is $O(s)$ perturbation of the matrix ${\mathcal J}_{{\mathcal G}}(x^*,0;0)$,
$$
\lambda_1(s)\cdots\lambda_{m+n-1}(s)=\det\widehat{A}\det\widehat{Q}+O(s)\;\;\mbox{and}\;\;\lambda(s)=O(s).
$$
This implies (\ref{16}).

Applying Lemma \ref{L13a} to the matrix ${\mathcal J}_{{\mathcal G}}(x(s),z(s);s)$ and using that $\widetilde{z}(s)=O(s^2)$ and $z_n,Q_c,Q_r,Q_n=O(s)$ we obtain
\begin{eqnarray*}
\det{\mathcal J}_{{\mathcal G}}(x(s),z(s);s)
 =\det\Big(\begin{bmatrix}
A&D_{\widetilde{z}}F\\
0 & \widehat{Q}\\
\end{bmatrix}+O(s)\Big)
\Big(\partial_{z_n}(z_nQ_n)-z_n\nabla_xQ_nA^{-1}\partial_{z_n}F+O(s^2)\Big)
\end{eqnarray*}
Therefore
$$
\det{\mathcal J}_{{\mathcal G}}(x(s),z(s);s)=\det A\det\widehat{Q}\Big(\partial_{z_n}(z_nQ_n)-z_n\nabla_xQ_nA^{-1}\partial_{z_n}F\Big)+O(s^2).
$$
Finally from equation (\ref{11x}) it follows that $Q_n(x(s),z(s);s)=O(s^2)$ therefore
$$
\det{\mathcal J}_{{\mathcal G}}(x(s),z(s);s)=-z_n\omega\det A\det\widehat{Q}+O(s^2).
$$
From (\ref{16}) it follows that
$$
\lambda(s)=-z_n\omega+O(s^2).
$$
Now by using (\ref{15}), we arrive at (\ref{162}).
\end{proof}

\bibliographystyle{plain}%

\end{document}